\documentclass{amsart}

\makeatletter
\@namedef{subjclassname@2020}{%
  \textup{2020} Mathematics Subject Classification}
\makeatother

\usepackage{mathpazo}
\usepackage[mathpazo]{flexisym}
\usepackage{breqn}
\usepackage{hyperref}

\usepackage{verbatim}
\usepackage{graphicx}
\usepackage{amsfonts,amsmath,latexsym,amssymb,amsthm}
\usepackage[T1]{fontenc} 
\usepackage{geometry}
\newcounter{theoremcounter}
\newcounter{lemmacounter}

\newcounter{dummycounter}

\newcounter{quescounter}

\newcounter{emptycounter}
\newcounter{defcounter}
\newcounter{probcounter}

\newtheorem{theorem}[theoremcounter]{Theorem}

\newtheorem{question}[quescounter]{Question}

\newtheorem{lemma}[lemmacounter]{Lemma}

\newtheorem{remark}{Remark}
\newtheorem{definition}[defcounter]{Definition}
\newtheorem{problem}[probcounter]{Problem}

\newcounter{eqncounter}

\numberwithin{equation}{eqncounter}

\def\IC{\mathbb C}
\def\IZ{\mathbb Z}
\def\IN{\mathbb N}

\def\IP{\mathbb P}
\def\IQ{\mathbb Q}

\def\p{\mathfrak{p}}
\def\q{\mathfrak{q}}
\def\ga{\gamma}

\renewcommand{\vec}[1]{\mbox{\boldmath$#1$}}
\def\Oseen{{\mathcal{O}}}
\def\A{{\mathfrak{A}}}
\def\C{{\mathfrak{C}}}

\def\D{{\mathfrak{D}}}
\def\B{{\mathfrak{B}}}

\let\rho\varrho
\def\ta_0{\tau}

\def\g0{\ta_0\sigma(\C_0^{-1}\D)}
\def\G0{\tau\Lambda(\D)}

\def\Qbar{\overline{\IQ}}

\def\vz{\vec{z}}

\def\v0{\vec{0}}
\def\d_1{\kappa}

\def\Gal{\mathop{{\rm Gal}}\nolimits}
\def\tors{\mathop{{\rm tors}}\nolimits}
\def\Fix{\mathop{{\rm Fix}}\nolimits}

\title{On the Northcott property for infinite extensions}
\author{Martin Widmer}

\begin{document}



\begin{abstract}
We start with a brief survey on the Northcott property for subfields of the algebraic numbers $\Qbar$. 
Then we introduce a new criterion for its validity (refining the author's previous criterion), 
addressing a problem of Bombieri. We  show that Bombieri and Zannier's theorem, stating that the maximal abelian extension
of a number field $K$ contained in $K^{(d)}$ has the Northcott property, follows 
very easily from this refined criterion.  Here $K^{(d)}$ denotes the composite field of all extensions of $K$ of degree
at most $d$.
%We give a short and simple proof of Bombieri and Zannier's theorem that the maximal abelian extension
%of a number field $K$ contained in $K^{(d)}$ (the composite field of all extensions of $K$ of degree
%at most $d$) has the Northcott property. Our proof is based on a refinement of a simple criterion
%from 2011 that uses a height lower bound of Silverman. 
\end{abstract}

\maketitle

\section{Introduction}

%Let $H(\cdot)$ be the absolute logarithmic Weil height on the algebraic numbers $\Qbar$.  
Heights are an important tool in Diophantine geometry to study the distribution of algebraic points on algebraic  varieties, and  
in arithmetic dynamics  to study  preperiodic points under  endomorphisms of algebraic varieties.
There are various different heights but  the most standard one is probably the Weil height on $\IP^n$.
However, there common fundamental property is that there are only finitely many points of bounded height
over a given number field. To which fields of infinite degree does this finiteness property extend?
This is the question we are concerned with in this article.

All algebraic field extensions of $\IQ$  are considered subfields of some fixed algebraic closure $\Qbar$.
Let $K$ be a number field, and for $P=(\alpha_0:\cdots:\alpha_n)\in \IP^n(K)$, with representative $(\alpha_0,\ldots,\alpha_n)\in K^{n+1}$, let 
\begin{alignat*}1
H(P)=\prod_{v\in M_K}\max\{|\alpha_0|_v,\ldots,|\alpha_n|_v\}^{\frac{d_v}{[K:\IQ]}}
\end{alignat*}
be the absolute multiplicative Weil height of $P$. Here $M_K$ denotes the set of places of $K$.  
For each place $v$ we choose the unique representative $|\cdot|_v$ that either extends the usual Archimedean absolute value
on $\IQ$ or a usual $p$-adic absolute value on $\IQ$, and $d_v=[K_v:\IQ_v]$ denotes the local degree at $v$. 
A standard reference for heights is \cite{BG}. We use $\IN=\{1,2,3,\ldots\}$  for the set of positive natural numbers.

The unique prime factorisation of $\IZ$ implies that $\prod_{M_\IQ}|\alpha|^{d_v}_v=1$ for every non-zero 
$\alpha\in \IQ$. This identity is known as the product formula  and extends 
to arbitrary number fields $K$ \cite[Proposition 1.4.4]{BG}. Consequently,  the value of the height is independent of the representative $(\alpha_0,\ldots, \alpha_n)$
and thus defines a genuine function on $\IP^n(K)$. Choosing a representative of $P$ with a coordinate equal to $1$  shows that $H(P)\geq 1$. 
The fundamental identity $\sum_{M_K}d_v=[K:\IQ]$, valid for every number field (cf.  \cite[Corollary 1.3.2]{BG}), shows that the  
height $H(P)$ is also independent from the number field $K$ containing the coordinates of $P$.
Hence,  $H(\cdot)$ is a well-defined  function on  $\IP^n(\Qbar)$. 
D. G. Northcott \cite[Theorem 1]{Northcott50} proved the following simple but  important result.
\begin{theorem}[Northcott, 1950]\label{thm: 0}
Given a number field $K$, $n\in \IN$,  and $X\geq 1$, there are only a finite number of points $P$ in $\IP^n(K)$ such that $H(P)\leq X$.
\end{theorem}
%Let $P=(\alpha_0:\cdots:\alpha_n)\in \IP^n(\Qbar)$ with a coordinate $\alpha_j\neq 0$.
%Then we obviously have  $H(P)\geq \max_{i} H((\alpha_j:\alpha_i))$.
%Consequently, Theorem \ref{thm: 0} holds true for a given field $K\subseteq \Qbar$ if and only if 
%it holds for  $n=1$.\\

For $P=(1:\alpha_1:\cdots:\alpha_n)\in \IP^n(\Qbar)$ we obviously 
have  $H(P)\geq \max_i H((1:\alpha_i))$.
Consequently, Theorem \ref{thm: 0} holds true for a given field $K\subseteq \Qbar$ if and only if 
it holds for  $n=1$. We define the height $H(\alpha)$ of an algebraic number $\alpha$
to be $H((1:\alpha))$, and so we are led to the following notion,
formally introduced in 2001 by Bombieri and Zannier \cite{BoZa}.
\begin{definition}[Northcott property]\label{def: N}
A subset  $S$ of $\Qbar$ has the \emph{Northcott property} (or shorter: \emph{Property (N)})
if $$\{\alpha\in{S}; H(\alpha)\leq X\}$$ is finite for every $X\geq 1$. 
\end{definition}
Theorem \ref{thm: 0} was merely an intermediate step in Northcott's seminal work \cite{Northcott50} from 1950 to show that
for any morphism $f: \IP^n\to \IP^n$ of algebraic degree at least $2$ and defined over a number field $K$ 
there are only finitely many preperiodic points in $\IP^n(K)$ under $f$. His proof also shows that one can replace number field
by any field with Property (N).

Another somewhat surprising application of Property (N) builds on work of J. Robinson from 1962.
It has been observed by Vidaux and Videla
\cite{ViVi16} that  her work \cite{Robinson1962} implies the undecidability of each
ring of  totally real algebraic integers with Property (N). This connection was further exploited in \cite{Martinez2020} and
in \cite{Springer2020}.

These two applications extend interesting properties of number fields to fields with Property (N),  suggesting that Property (N) fields
behave similarly as number fields. However, this view was shattered by Fehm's discovery \cite[Proposition 1.2]{Fehm}  that some fields with Property (N) 
are pseudo algebraically closed (PAC). 

 Next we discuss  two arithmetic properties with respect to which \emph{all} fields of infinite degree with Property (N)  behave radically different from
number fields. 

Gaudron and R\'emond \cite{GaudronRemondSiegel} introduced the notion of a Siegel field, which is a subfield of $\Qbar$
over which Siegel's Lemma holds true (cf. \cite[$(\ast)$ on p.189]{GaudronRemondSiegel}). It is classical that number fields
are Siegel fields, and work of Zhang \cite{PosLineBundVar}, and independently  of Roy and Thunder \cite{79}, shows that $\Qbar$ is also a Siegel field.
A priori it is not easy to find counterexamples but Gaudron and R\'emond \cite[Corollaire 1.2]{GaudronRemondSiegel} proved that a field of
infinite degree with Property (N) cannot be a Siegel field.

A very recent paper of Daans, Kala and Man  \cite{DaansKalaMan} investigates the existence of universal quadratic forms over totally real fields of infinite degree.
Whereas it is well-known that for totally real number fields such a form always exists, the existence of a universal quadratic  form
over a given  totally real field of infinite degree is not clear at all.
However, they prove \cite[Theorem 1.2]{DaansKalaMan} that such a form cannot  exists if the field  has infinite degree and Property (N).\\

A point $P=(\alpha_0:\cdots:\alpha_n)\in \IP^n(\Qbar)$ defines a number field $\IQ(\alpha_i/\alpha_j; \alpha_j\neq 0)$, 
and the degree of $P$ is the degree of this number field.
To prove Theorem \ref{thm: 0} Northcott proved the stronger result \cite[Lemma 2]{Northcott50} that for any given $d\in \IN$ and $X\geq 1$ there are only  finitely many points $P\in \IP^n(\Qbar)$ of degree $d$ and height $H(P)$ at most $X$. The latter is a direct consequence
%by showing that the set of algebraic numbers of degree at most $d$ has the Northcott property. 
%And this statement in turn  is what nowadays is usually understood as ``Northcott's Theorem'' (cf. \cite[Theorem 1.6.8]{BG}).
of what  nowadays is usually understood as ``Northcott's Theorem'' (cf. \cite[Theorem 1.6.8]{BG}).
\begin{theorem}[Northcott's Theorem]\label{thm: 3}
Let  $d\in \IN$, then the set $\{\alpha\in \Qbar; [\IQ(\alpha):\IQ]\leq d\}$ has Property (N).
\end{theorem}
%It is worthwhile to point out that 
Northcott's Theorem  already  implies the existence of fields of infinite degree with Property (N).
Indeed, let $K$ be a number field and let $X\geq 1$ be given.
Any two distinct quadratic extensions of $K$ only intersect in $K$, and there are infinitely many such extensions. Hence,  there must be one 
whose elements outside of $K$ all have height bigger than $X$. Constructing an infinite tower $\IQ=K_0\subset K_1\subset K_2\subset \cdots$
where we choose a quadratic extension $K_{i+1}$ of $K_i$ whose elements outside of $K_i$ all have height larger than $i$ say, yields an
infinite extension $L=\cup_i K_i$ with Property (N).

Dvornicich and Zannier \cite{DvZa} observed that Northcott's Theorem remains true when replacing the ground field $\IQ$ by any field with Northcott property, i.e., 
if $L$ is a field with Property (N) and $d\in \IN$, then the set 
\begin{alignat*}1
\{\alpha\in \Qbar; [L(\alpha):L]\leq d\}
\end{alignat*}
also has Property (N). In particular, Property (N) is preserved under finite field extensions. However, it is not always preserved under
taking Galois closure over $\IQ$, or taking compositum of two fields (cf.  \cite[Theorem 5 ]{WidmerPropN}).\\

Bombieri and Zannier \cite{BoZa} were the first\footnote{It is worthwhile mentioning that Julia Robinson \cite{Robinson1962}  in 1962 proved that the ring of integers of $\IQ(\sqrt{n}; n\in \IN)$ has the ``Northcott property'' with respect to the house (instead of Weil height), and deduced from this that  $\IN$ is first order definable in this ring.} authors that studied the Northcott property for infinite field extensions of $\IQ$.
In view of Northcott's Theorem it is very appealing to consider the field $\IQ^{(d)}$ generated over $\IQ$ by all algebraic numbers of degree at most $d$.
Bombieri and Zannier \cite{BoZa} raised the following question.
\begin{question}[Bombieri and Zannier,  2001]\label{quest: 1}
Let $d\in \IN$. Does $\IQ^{(d)}$ have Property (N)?
\end{question}
%It should be mentioned that 
There is a  whole zoo of properties  for subfields of $\Qbar$ (including the properties $(P), (SP), (\overline{P}), (R), (\overline{R}), (K)$, see \cite{PolyMapp, Liardet71}; and $(SB), (USB)$, see \cite{FiliMiner, Pottmeyer}) in arithmetic dynamics, that are all implied by Property (N) (cf. \cite{CheccoliWidmer, Pottmeyer}).
For some of these properties the analogue of  Question \ref{quest: 1} was posed, explicitly\footnote{Narkiewicz \cite[Problem 10 (i)]{unsolvedproblems,Narkiewicz1963} conjectured that $K^{(d)}$ has (P) for all $d$. Further, for various pairs of  these properties it was asked whether they are equivalent to each other, cf. \cite{PolyMapp, CheccoliWidmer}}
or implicitly.
We will  not discuss any of these more exotic properties but let us mention that 
Pottmeyer \cite[Theorem 4.3]{Pottmeyer} showed that $\IQ^{(d)}$ has the properties (USB) and (P) (solving a conjecture of Narkiewicz from 1963). However, (USB) and (P) are 
both strictly weaker than (N), as shown in \cite[Proposition 1.3]{Fehm} and in  \cite[Theorem 3.3]{DvZa}  respectively.\\

Question \ref{quest: 1} is still open 
but a remarkable step was already made in \cite{BoZa}.
For $d\in \IN$ and  $K$ a number field
we write $K^{(d)}$ for the composite field of all extensions of $K$  of degree at most $d$. Then $K^{(d)}/K$ is a Galois extension,
generated over $K$ by all algebraic numbers of relative degree $[K(\alpha):K]$ at most $d$.
Let $K^{(d)}_{ab}$ be the composite field of all abelian extensions $F/K$ with $F\subset K^{(d)}$.
Then $K^{(d)}_{ab}$  is the maximal abelian subextension of $K^{(d)}/K$. 
If $d\geq 2$ then $\IQ(\sqrt{n}; n\in \IZ)\subset K^{(d)}_{ab}\subset K^{(d)}$, and so $K^{(d)}_{ab}$ and $K^{(d)}$ both have infinite degree over $\IQ$, and thus also over $K$.

\begin{theorem}[Bombieri, Zannier 2001]\label{thm: 1}
Let $K$ be a number field and let $d\in \IN$.
The field  $K^{(d)}_{ab}$ has the Northcott property. In particular, $K^{(2)}$  has the Northcott property.
\end{theorem}
Taking $K=\IQ(\zeta_d)$ for a primitive $d$-the root of unity, and applying Theorem \ref{thm: 1} proves that  the field 
\begin{alignat}1\label{cor: BoZa}
\IQ(1^{1/d}, 2^{1/d}, 3^{1/d}, 4^{1/d},\ldots)
\end{alignat}
has the Northcott property. 

Theorem \ref{thm: 1} is  a very interesting result for its own sake but it also has interesting applications. 
Specifically, to list  some of the recent applications, Theorem \ref{thm: 1} was used:
\begin{itemize}
\item   in \cite{ViVi16} to show that the maximal totally real subfield of $K^{(d)}_{ab}$ is undecidable,
in \cite{Springer2020} to show that  $\IQ^{(d)}_{ab}$ is undecidable,
and in \cite{Martinez2020} as one of the ingredients that led the authors conjecture that $K^{(d)}$ is undecidable (proved for $\IQ^{(2)}$
in the same paper).
\item  in \cite{DaansKalaMan}  to deduce that if $L$ is a totally real subfield of $K^{(d)}_{ab}$ of infinite degree, then
no universal quadratic form exists over $L$. In particular, this holds if  $L\subset \IQ^{[d]}$ and $d$ is a prime or a prime square,
where $\IQ^{[d]}$ denotes the
compositum of all totally real Galois fields of degree exactly $d$ over $\IQ$.
\item in \cite[Corollary 1]{CheccoliDill2023} to prove that if $K$ is a number field,  $A$ is an abelian variety defined over $K$, and 
$K(A_{\tors})$ is  the minimal field extension of $K$ over which all torsion points of $A$ are defined, then 
each subfield of $K(A_{\tors})$ which is Galois over $K$, and whose Galois group has finite exponent, has the Northcott property.
\end{itemize}

An abelian extension $L/\IQ$ lies in $\IQ^{(d)}$ for some $d$ if and only if its Galois group has finite exponent (cf. \cite[Theorem 1]{Checcoli}). 
As pointed out in  \cite[Section 5]{CheccoliDill2023}  this remains true when replacing the ground field $\IQ$ with an arbitrary number field $K$.
Therefore Theorem \ref{thm: 1} gives a purely Galois theoretic criterion for the Northcott property of a field, i.e., every abelian extension
of a number field $K$ with finite exponent has the Northcott property.

However, the restriction to abelian extensions (and finite exponent) in Theorem \ref{thm: 1} is very rigid and rules out many interesting
examples. 
In the survey article \cite[p. 52]{BomSur} Bombieri states:
\emph{``It remains an open problem to determine whether the Northcott property
holds for $K^{(d)}$ if $d\geq 3$ and, more generally, to determine workable
conditions for its validity.''}
In this paper we are particularly concerned with the second part of Bombieri's statement.

\begin{problem}[Bombieri,  2009]\label{prob: 1}
Determine  workable conditions for the validity of the Northcott property for subfields of $\Qbar$.
\end{problem}

In 2011 the author \cite{WidmerPropN} gave a criterion which is robust and often easy to apply.
For an extension $M/K$ of number fields we write  $D_{M/K}$ for the relative discriminant,
and we write $N_{K/F}(\cdot)$ for the norm from $K$ to $F$. If $F=\IQ$ and $\A$ is a non-zero  ideal in the ring of integers $\Oseen_K$
of $K$ then we interpret $N_{K/F}(\A)$ as the unique positive rational integer that generates the principle ideal $N_{K/F}(\A)$.
\begin{theorem}[{\cite[Theorem 3]{WidmerPropN}}]\label{thm: 2}
Let $K$ be a number field, let $K=K_0\subsetneq K_1\subsetneq K_2\subsetneq....$ be a nested sequence of finite extensions and set $L=\bigcup_{i}K_i$. Suppose that
\begin{alignat}1\label{reldiscond0}
\inf_{K_{i-1}\subsetneq M \subset K_i}\left(N_{K_{i-1}/\IQ}(D_{M/K_{i-1}})\right)^{\frac{1}{[M:K_0][M:K_{i-1}]}}\longrightarrow \infty
\end{alignat}
as $i$ tends to infinity where the infimum is taken over all intermediate fields $M$ strictly larger than $K_{i-1}$.
Then  the field $L$ has the Northcott property.
\end{theorem}

Theorem \ref{thm: 2} implies the following refinement of (\ref{cor: BoZa}). 
Let $K$ be a number field,
let $p_1<p_2<p_3<...$ be a sequence of positive primes
and let $d_1,d_2,d_3,...$ be a sequence of positive integers.
Then  the field 
$$K(p_1^{1/d_1},p_2^{1/d_2},p_3^{1/d_3},...)$$ 
has the Northcott property if and only if
$\log p_i/d_i\longrightarrow \infty$ as $i$ tends to infinity.
The fact that every direct product of finite solvable groups can be realised over $\IQ$ by a Galois extension with Property (N)
can also easily be deduced from Theorem \ref{thm: 2}  (cf. \cite[Theorem 4]{CheccoliWidmer}).
Fehm's aforementioned construction of PAC fields with Property (N) also used   Theorem \ref{thm: 2}.
And finally, Theorem \ref{thm: 2} allows to construct fairly large non-abelian subfields of $\IQ^{(d)}$ with Property (N)
(cf. \cite[Corollaries 3, 4, and 5]{WidmerPropN}), providing another result on Question \ref{quest: 1}.

Theorem \ref{thm: 2} is based on a  fundamental height lower bound of Silverman \cite[Theorem 2]{Silverman}.
Here we give only a simplified version sufficient for our purposes.
Let $\alpha\in \Qbar$, let $F$  be a number field, let $K=F(\alpha)$, $m=[F:\IQ]$, and $d=[K:F]$. Then
\begin{alignat}1\label{ineq: Sil}
 H(\alpha)\geq \frac{1}{2}N_{F/\IQ}(D_{K/F})^{\frac{1}{2md^2}}.
\end{alignat}
Using the optimal choice of $F$ for given $\alpha$ to maximise the right hand-side in (\ref{ineq: Sil}) plays an important role in our results.
For the convenience of the reader we  will give a proof of inequality (\ref{ineq: Sil}) in Section \ref{sec: Sil}.

%How does one prove Theorem \ref{thm: 2}?
%Given $X\geq 1$, inequality (\ref{ineq: Sil})  tells us that all elements of $L$ of height at most $X$ must lie in some fixed number field $K_i$ (of course $i$ depends on $X$),
%and thus, by Northcott's Theorem, the field $L$ has the Northcott property. 

Obviously Theorem \ref{thm: 2} does not follow from Theorem \ref{thm: 1}. 
How does one prove Theorem \ref{thm: 2}? Let $\alpha\in L$ be of height at most $X$, and let $K_{i_0}$ be the maximal
field not containing $\alpha$.
Applying (\ref{ineq: Sil}) with $F=K_{i_0}$, and using (\ref{reldiscond0}), shows that  $i_0$ is bounded from above in terms of $X$ and $L$,
and thus, by Northcott's Theorem, the field $L$ has the Northcott property. 

However, the choice $K_{i_0}$ for the ground field $F$ can be far from optimal, and so we do not use the full force of (\ref{ineq: Sil}). 
Therefore, Theorem \ref{thm: 2} does not seem strong enough to deduce Theorem \ref{thm: 1} either.\\

The aim of this short note is to provide a  refined criterion, using the full force of (\ref{ineq: Sil}),
that easily implies Theorem \ref{thm: 2} and Theorem \ref{thm: 1}.
To this end we introduce the 
 following invariant for an extension of number fields $M/K$:
\begin{alignat}1\label{def: ga}
\ga(M/K)=\sup_{K\subset F} \left(N_{F/\IQ}(D_{MF/F})\right)^{\frac{1}{[MF:\IQ][MF:F]}},
\end{alignat}
where the supremum runs over all number fields $F$ containing $K$, and $MF$ denotes the composite field of $M$ and $F$.
We can now state a more powerful version of the criterion given in  Theorem \ref{thm: 2}.

\begin{theorem}\label{prop:1}
Let $K$  be a number field,  and let $L$ be an infinite algebraic field extension of $K$. Suppose that
\begin{alignat*}1
\liminf_{K\subset M\subset L}\ga(M/K)=\infty,
\end{alignat*}
where    $M$ runs over all number fields in $L$ containing $K$. Then $L$ has the Northcott property.
\end {theorem}
\begin{proof}
Suppose that $L$ does not have the Northcott property. Thus there exists $X\geq 1$ and
a sequence $(\alpha_i)_i$ of pairwise distinct elements in $L$ with $H(\alpha_i)\leq X$ for all $i$. By Northcott's Theorem the degrees of $M_i=K(\alpha_i)$ must tend to infinity. After passing to a subsequence we can assume all the $M_i$
are distinct. Note that $M_iF=F(\alpha_i)$ for each $F$ that contains $K$.
We apply inequality (\ref{ineq: Sil}) to get 
\begin{alignat*}1
4X^2&\geq \liminf_i (2H(\alpha_i))^2\geq \liminf_i \left(\sup_{K\subset F} N_{F/\IQ}(D_{M_iF/F})^{\frac{1}{[M_iF:\IQ][M_iF:F]}}\right)\\
&\geq \liminf_{K\subset M\subset L}\left(\sup_{K\subset F} N_{F/\IQ}(D_{MF/F})^{\frac{1}{[MF:\IQ][MF:F]}}\right)=\liminf_{K\subset M\subset L}\ga(M/K).
\end{alignat*}
\end{proof}

Theorem \ref{prop:1} implies\footnote{Let (M_j) be a sequence of distinct fields with $K\subset M_j\subset L$ and $\gamma(M_j/K)<X$, and let $i=i(j)$ be minimal with $M_j\subset K_i$. Set $M'_j=K_{i-1}M_j$ so that $K_{i-1} \subsetneq M'_j\subset K_i$.
The choice  $F=K_{i-1}$ on the right-hand side of (\ref{def: ga}) shows that (1.2) has a bounded subsequence.} Theorem \ref{thm: 2}, but why does it also imply Theorem \ref{thm: 1}, and how does this proof differ from the original one 
in \cite{BoZa}? We will discuss these questions  in detail in Section \ref{sec: proof}.\\

Are there any known criteria for Property (N) for field extensions of infinite degree that we have not mentioned so far?
%besides Theorem \ref{thm: 1} and  Theorem \ref{thm: 2}?
The author is only aware of one such criterion.
%In \cite{BoZa} the authors also proved a second criterion besides Theorem \ref{thm: 1}. 
Let $L/\IQ$ be a Galois extension and let $S(L)$ be the set of rational
primes for which $L$ can be embedded in a finite extension of $\IQ_p$. For $p\in S(L)$ let $e_p$ and $f_p$ be the ramification index and the 
inertia degree above $p$. Bombieri and Zannier \cite[Theorem 2]{BoZa} proved that 
\begin{alignat}1\label{crit: 3}
\liminf_{\alpha\in L}H(\alpha)\geq \exp\left(\frac{1}{2}\sum_{p\in S(L)}\frac{\log p}{e_p(p^{f_p}+1)}\right).
\end{alignat}
In particular, $L$ has the Northcott property whenever the sum on the right hand-side of (\ref{crit: 3}) diverges.
The above criterion does not seem very workable.  Bombieri and Zannier  asked whether
this sum can diverge for infinite extensions but considered this unlikely. However, it was shown by Checcoli and Fehm \cite{CheccoliFehm} in 2021
that there are Galois extensions $L/\IQ$ of infinite degree for which the above sum diverges, and even such extensions
for which neither Theorem \ref{thm: 1} nor Theorem \ref{thm: 2} applies, so it constitutes an independent criterion for the Northcott property,
albeit one for which natural examples still need to be found.

\section{Silverman's inequality}\label{sec: Sil}
In this section we give a proof of Silverman's inequality (\ref{ineq: Sil}). 
For the special case $F=\IQ$ a very simple proof was given by Roy and Thunder \cite[Lemma 1 and 2]{8}. 
We extend the argument in  \cite{8} to arbitrary ground fields $F$, providing a slightly different  proof from Silverman's original one in \cite{Silverman}.
Yet another proof of Silverman's inequality was given by Ellenberg and Venkatesh \cite[Lemma 2.2]{EllVentorclass}.

We first fix the notation and recall some basic facts. Let $F$ be a number field of degree $m$,
let $K/F$ be a field extension  of degree $d$, and let 
$\sigma_1,\ldots,\sigma_d: K\to K^{(G)}$ be the $d$ distinct  field homomorphisms of $K$ to  the Galois closure $K^{(G)}$ of $K/F$, fixing $F$.
Let $(z_1,\ldots,z_d)$ be a $d$-tuple of elements in $K$. Then $D_{K/F}(z_1,\ldots,z_d)=\det[\sigma_i(z_j)]^2$, and for a non-zero ideal $\A$ in $\Oseen_K$
the discriminant  $D_{K/F}(\A)$ is the ideal in $\Oseen_F$ generated by the numbers $D_{K/F}(z_1,\ldots,z_d)$
 as the tuples $(z_1,\ldots,z_d)$ run over all $F$-bases of $K$ and each basis element is contained in $\A$. In particular, $D_{K/F}(\A)$ divides
the principle ideal in $\Oseen_F$ generated by $D_{K/F}(z_1,\ldots,z_d)$  for each such tuple $(z_1,\ldots,z_d)$ (see \cite[III, \S 3]{13}).
Recall that we write $D_{K/F}$ for $D_{K/F}(\Oseen_K)$. We will use the basic identity (cf. \cite[III, \S 3, Proposition 13]{12})
\begin{alignat}1\label{eq: disc}
D_{K/F}(\A)=D_{K/F}N_{K/F}(\A)^2.
\end{alignat}
\begin{lemma}[Silverman, 1984]\label{lem: Silverman}
Let $F$  be a number field of degree $m$.  Let $\alpha\in \Qbar\backslash F$,  set $K=F(\alpha)$, and $d=[K:F]$. Then
\begin{alignat*}1
 H(\alpha)\geq d^{-\frac{1}{2(d-1)}}N_{F/\IQ}(D_{K/F})^{\frac{1}{2md(d-1)}}.
 \end{alignat*}
\end {lemma}
\begin{proof}
Choose $\omega_0,\omega_1\in \Oseen_K$ such that $\omega_0\neq 0$ and $\alpha=\omega_1/\omega_0$. 
For $1\leq j\leq d$ let $z_j=\omega_0^{d-j}\omega_1^{j-1}$,  so that $P=(1:\alpha:\cdots:\alpha^{d-1})=(z_1:\cdots: z_d)\in \IP^{d-1}(K)$
and $H(\alpha)^{d-1}=H(P)$. We will bound
$$H(P)^{2md}=\prod_{v\nmid \infty}\max_j\{|z_j|_v\}^{2d_v}\prod_{v\mid \infty}\max_j\{|z_j|_v\}^{2d_v}$$
from below.
Note that $z_1,\ldots,z_d$ is an integral $F$-basis of $K$.
Let $\A=\sum_j z_j\Oseen_K$ be the ideal in $\Oseen_K$   generated by  the $z_j$. 
For the non-Archimedean places of $K$ we have
\begin{alignat*}1%\label{eq: non-arch}
\prod_{v\nmid \infty}\max_j\{|z_j|_v\}^{2d_v}=N_{K/\IQ}(\A)^{-2}. 
\end{alignat*}
For each  embedding $\tau:F \to \IC$ 
we choose an extension $\tilde{\tau}:K^{(G)}\to \IC$ of $\tau$ to  $K^{(G)}$.
Then  the $d$ distinct maps $\tilde{\tau}\circ\sigma_i:K\to \IC$  are precisely the $d$ embeddings of $K$ that extend $\tau$. 
Ranging over all embeddings $\tau$ of $F$ gives the full set of embeddings of $K$.
Hence, for  the Archimedean places  of $K$ we get  
\begin{alignat*}1%\label{eq: arch}
\prod_{v\mid \infty}\max_j\{|z_j|_v\}^{2d_v}=\prod_{\tau}\prod_{i=1}^d\max\{|\tilde{\tau}\circ\sigma_i (z_1)|,\ldots, |\tilde{\tau}\circ\sigma_i (z_d)|\}^2.
 \end{alignat*}
Writing $\vz_{\tau,i}$ for the complex row vector $(\tilde{\tau}\circ\sigma_i (z_1),\ldots, \tilde{\tau}\circ\sigma_i (z_d))$, and applying Hadamard's inequality 
yields
\begin{alignat*}1
\prod_{i=1}^d\max\{|\tilde{\tau}\circ\sigma_i (z_1)|,\ldots, |\tilde{\tau}\circ\sigma_i (z_d)|\}^2&\geq d^{-d}\prod_{i=1}^d |\vz_{\tau,i}|^2\geq d^{-d}|\det[\tilde{\tau}\circ\sigma_i (z_j)]^2|
=d^{-d} |\tilde{\tau}(\det[\sigma_i (z_j)]^2)|=d^{-d} |\tau(\det[\sigma_i (z_j)]^2)|,
 \end{alignat*}
 where in the last step we used that $\det[\sigma_i (z_j)]^2=D_{K/F}(z_1,\ldots,z_d)$ lies in $F$.
Taking the product over all $\tau$, and using that  $D_{K/F}(\A)$ divides the ideal generated by $\det[\sigma_i (z_j)]^2$ in $\Oseen_F$, yields
\begin{alignat*}1
\prod_{v\mid \infty}\max_j\{|z_j|_v\}^{2 d_v}\geq d^{-md} N_{F/\IQ}(D_{K/F}(\A)).
 \end{alignat*}
Now we use  (\ref{eq: disc}), and that $N_{F/\IQ}(D_{K/F} N_{K/F}(\A))^2)=N_{F/\IQ}(D_{K/F}) N_{K/\IQ}(\A)^2$ to get
 \begin{alignat*}1
H(\alpha)^{2md(d-1)}=H(P)^{2md}\geq d^{-md} N_{F/\IQ}(D_{K/F}),
 \end{alignat*}
which proves the claim.
\end{proof}

\section{Theorem \ref{prop:1} implies Bombieri and Zannier's Theorem \ref{thm: 1}}\label{sec: proof}
In this section we show that Theorem \ref{prop:1} gives a short and straightforward proof of Theorem \ref{thm: 1}.
We also compare this new proof with the original one from \cite{BoZa}. 
Both proofs have a common part which we extract and formulate below as a separate lemma.
\begin{lemma}[Bombieri and Zannier \cite{BoZa}]\label{lem: basic}
Let $d\in \IN$, let $K$ be a number field,
and let $M$ be a number field with $K\subset M\subset K^{(d)}_{ab}$. 
Then $p_M$, the largest prime that ramifies in $M$,
tends to infinity as $M$ runs over all such intermediate fields $M$.
Further, if $p>d$ is prime and  $\B$ is a prime ideal in $\Oseen_M$ above $p$ and $\p=\B\cap K$,
then   the ramification index $e(\B/\p)$ divides $d!$.
\end{lemma}
\begin{proof}
We follow  Bombieri and Zannier's argument from \cite{BoZa}.
Let $M$ be a number field with $K\subset M\subset K^{(d)}_{ab}$. Then $M/K$ is an abelian extension of exponent\footnote{If $K_1, K_2$ are two finite Galois extensions of $K$ then $\sigma \to (\sigma_{|_{K_1}},\sigma_{|_{K_2}})$ induces an injective group homomorphism from $\Gal(K_1K_2/K)$ to $\Gal(K_1/K)\times\Gal(K_2/K)$. 
This implies that for each Galois extension $M/K$ with $M\subset K^{(d)}$ the Galois group $\Gal(M/K)$ has exponent dividing $d!$, and no prime $p>d$ divides the order of $\Gal(M/K)$.} 
dividing  $d!$, and thus $\Gal(M/K)$ is isomorphic to a direct product $A_1\times \cdots\times A_r$ of cyclic groups of order dividing $d!$. 
Therefore $M$ can be written as composite field of extensions $E$ of $K$ of degree at most $d!$. Indeed, let $\varphi:A_1\times \cdots\times A_r\to \Gal(M/K)$ be an isomorphism and let $E_i=\Fix(\varphi (H_i))$ where $H_i$ is
the subgroup that picks the trivial group in the $i$-th component and the full $A_j$ in all other components. By the Galois-correspondence we have $\Gal(M/E_1\cdots E_r)=\cap_i \varphi(H_i)=\varphi(\cap_i H_i)=\{id\}$. Hence, $M=E_1\cdots E_r$. Now the largest power of a prime dividing the discriminant of $E$ can be bounded solely in terms of $K$ and $d$ (cf. \cite[Theorem B.2.12]{BG}).
Thus, by Hermite's Theorem,   $p_M$,  the largest prime that ramifies in $M$,
tends to infinity as $M$ runs over all such intermediate fields $M$.

For the second claim note that the inertia group $I(\B/\p)$ is a subgroup of $\Gal(M/K)$, and so its order is not divisible by $p$, whenever $p>d$ is prime. Since the ramification index $e(\B/\p)$  is equal to the order of $I(\B/\p)$ it follows that $\p$ is tamely ramified in $M$. Hence (cf. \cite[B.2.18 (e)]{BG}), $I(\B/\p)$ is cyclic,
and thus $e(\B/\p)$ divides $d!$.
\end{proof}

%\begin{remark}
%The second claim of Lemma \ref{lem: basic} holds even when $K^{(d)}_{ab}$ is replaced by $K^{(d)}$. This follows from 
%the following strong version of Abhyankar's Lemma (cf. \cite[Theorem 2.1]{ChabertHalberstadt}).
%Let $K_1,K_2$ be two finite extensions of the number field $K$, let $L=K_1K_2$,
%let $\B$ be a prime ideal in $\Oseen_L$ over the rational prime $p$, let $\p=\B\cap K$, and let $\p_i=\B\cap K_i$ (for $i=1,2$).
%If $p$ does not divide $e(\p_2/\p)$, then  $e(\B/\p)$ is the least common multiple of $e(\p_1/\p)$ and $e(\p_2/\p)$.
%\end{remark}

%\cite[Th. 3.9.1]{Stichtenoth} for function fields 

%The second claim of Lemma \ref{lem: basic} holds for every prime $p$, even when $K^{(d)}_{ab}$ is replaced by $K^{(d)}$. This follows from Abhyankar's Lemma
%and the following  fact:  if $K_1,K_2$ are two finite extensions of the number field $K$ and $K_1/K$  is Galois, then $e(\B/\p)|e(\p_1/\p)e(\p_2/\p)$,
%where $\p$ and $\B$ are prime ideals in $K$ and $L=K_1K_2$ respectively, and $\p_i=\B\cap K_i$ (for $i=1,2$). 
%
%The latter can be seen as follows (cf. \cite[(2)]{ChabertHalberstadt}). Since $K_1/K$  is Galois also $L/K_2$  is Galois. 
%Restricting the elements of $\Gal(L/K_2)$ to $K_1$ defines an injective homomorphism $\phi: \Gal(L/K_2) \to \Gal(K_1/K)$.
%The image of the inertia group $I(\B/\p_2)$ lies in the inertia group $I(\p_1/\p)$, and thus, $e(\B/\p_2)|e(\p_1/\p)$. Multiplying both sides with $e(\p_2/\p)$
%yields the claim.

Now let us show that Theorem \ref{prop:1} together with Lemma \ref{lem: basic} implies Theorem \ref{thm: 1}. 
\begin{proof}[Proof of Theorem \ref{thm: 1}]
Let $M$ be a number field with $K\subset M\subset K^{(d)}_{ab}$. Then $M$ is  abelian over $K$.  
By Lemma \ref{lem: basic} $p_M>|D_{K/\IQ}|+d$  for all but finitely many $M$, and thus we can assume $p_M$ is unramified in $K$ and $p_M>d$.  Therefore, one of the prime ideal divisors of $p_M\Oseen_K$, say $\p$, must ramify in $M$.
%Let   $M$ be a  number field in $K^{(d)}_{ab}$. Then $M$ is  abelian over $K$. Let $p=p_M$ be the largest prime that ramifies in $M$.
%Ignoring finitely many $M$ we can assume that $p_M$ is unramified in $K$. 
Let $\p\Oseen_M=(\B_1\cdots\B_g)^e$
be the  decomposition in $\Oseen_M$ with $\B_1,\ldots, \B_g$ distinct prime ideals.
%We  consider  the inertia groups ${I_{\B_i}}$  attached to the  prime ideal $\B_i$.
Let $T_i$ be the fixed field for the inertia group $I(\B_i/\p)$, and
%The prime ideal $\p$ is unramified in $T_i$
%and $\p_i=\B_i\cap T_i$ is totally ramified in $M$. It follows that $\p_i^{e-1}|D_{M/T_i}$, and that $f(\B_i/\p_i)=1$ so that  $f(\B_i/\p)=f(\p_i/\p)$ for  the residue degree.
let $\p_i=\B_i\cap T_i$. Then $e(\p_i/\p)=1$, and $e=e(\B_i/\p_i)=[M:T]$.
It follows that $\p_i^{e-1}|D_{M/T_i}$, and that $f(\B_i/\p)=f(\p_i/\p)$ for  the residue degree.
Now the $I(\B_i/\p)$ are conjugated to each other and since $M/K$ is abelian they are all  equal, and thus all the fixed fields $T_i$ are equal to $T$, say.
Therefore $(\p_1\cdots \p_g)^{e-1}|D_{M/T}$, which implies $p_M^{[M:K]/2}\leq N_{T/\IQ}(D_{M/T})$.
%Now $[M:T]=e$, and thus 
Choosing $F=T$ in (\ref{def: ga}) 
shows that $\ga(M/K)\geq p_M^{1/(2e[K:\IQ])}$ which, by Lemma \ref{lem: basic}, tends to infinity as $M$ runs over all number fields with $K\subset M\subset K^{(d)}_{ab}$. Applying Theorem \ref{prop:1} completes the proof.\\
\end{proof}

\begin{remark}
Alternatively, one can use the decomposition of $M$ as compositum of extensions $E$ of $K$ of degree at most $d!$ as in the proof of Lemma \ref{lem: basic}. 
Hence, $\p$ ramifies in at least one of the  fields $E$, and thus  $\p|D_{E/K}$. 
Since $D_{M/K}=D_{E/K}^{[M:E]}N_{E/K}(D_{M/E})$ we conclude  $\p^{[M:E]}|D_{M/K}$.
Since  $\p$ is unramified in $T$, and  $D_{M/K}=D_{T/K}^{[M:T]}N_{T/K}(D_{M/T})$
we get $\p^{[M:E]}| N_{T/K}(D_{M/T})$. Taking norms and using $[M:T]=e$ gives
$\ga(M/K)\geq  p^{\frac{1}{[K:\IQ]d!e}}$.
%\begin{alignat}1\label{ineq:3}
%\ga(M/K)\geq \left(N_{T/\IQ}(D_{M/T})\right)^{\frac{1}{[M:\IQ][M:T]}}\geq p_M^{\frac{[M:E]}{[M:\IQ][M:T]}}\geq p_M^{\frac{1}{[K:\IQ]d!e}}.
%\end{alignat}
\end{remark}

To compare we now discuss Bombieri and Zannier's original proof of  Theorem \ref{thm: 1}. 
%Recall that we consider the special case  $\IQ^{(d)}_{ab}$, although this is only a minor simplification.
We leave out some of the more technical details but the basic argument is as follows. 
We mostly use the notation of \cite{BoZa} 
(see also \cite[Theorem 4.5.4]{BG} for a slightly more detailed approach).

\begin{proof}[Proof of Theorem \ref{thm: 1} (after Bombieri and Zannier)]
By enlarging $K$ we can assume $K$ contains  a primitive $d!$-th root of unity.
Let $\alpha\in K^{(d)}_{ab}$ be of height at most $X$, and set  $L=K(\alpha)$. 
Then $L$ is abelian over $K$.
%, and the exponent of $\Gal(L/K)$ divides $d!$. 
Let $p>d$ be a prime unramified in $K$, let $v$ be a place in $K$ above $p$, and write $e$ for  the ramification index of $v$ in $L$.
Then $e$ divides $d!$ by Lemma \ref{lem: basic}.

%Since $p>d$ it follows that $p$ does not divide $\Gal(L/K)$, and thus each place in $L$ above $v$  is tamely ramified\footnote{Note that ``tamely ramified'' includes ``unramified''.},
%with cyclic inertia group of order $e$. This shows that $e$ divides $d!$.

Now set $\theta=p^{1/e}$. Then $L(\theta)$ is again an abelian extension of $K$.
Since $x^e-p\in K[x]$ is a $v$-Eisenstein polynomial it follows that $[K(\theta):K]=e$ and $v$ is totally ramified in $K(\theta)$.
By Abhyankar's Lemma  the ramification indices of the places in $L(\theta)$ above $v$  are again $e$.
As $\Gal(L(\theta)/K)$
is abelian the inertia groups of each place in $L(\theta)$ above $v$ are equal, and of size $e$. Let $U$ be their common fixed field,
so that $[L(\theta):U]=e$. Now $v$ is unramified in $U$ and totally ramified in $K(\theta)$ and thus $[U(\theta):U]=e$.
Hence, 
$L(\theta)=U(\theta)$, and thus 
$$\alpha=\beta_0+\beta_1 \theta+ \cdots + \beta_{e-1}\theta^{e-1}$$
for certain coefficients $\beta_i\in U$.
%Now computing the trace from $U(\theta)$ to $U$ of $\alpha\theta^{-j}$, and using standard height inequalities, allows one to get an upper bound for the height
%of $\gamma_j:=\beta_jp^{j/e}$ in terms of $d$ and $X$. 
Now  the trace from $U(\theta)$ to $U$ of $\alpha\theta^{-j}$ is the sum of the conjugates of $\alpha\theta^{-j}$ over $U$. 
It is not hard to see that this trace is also just $e\beta_j$.
Combining both, and using standard height inequalities, gives an upper bound for the height
of $\gamma_j:=\beta_jp^{j/e}$ in terms of $d$ and $X$. 

Let us now assume that  $1\leq j\leq e-1$ and $b_j\neq 0$. Let $u$ be a place in $U(\theta)$ above $v$, and let  $\q=\q_u$ be the corresponding prime ideal in
the ring of integers of  $U(\theta)$. Then the exact order to which $\q$ divides $b_j$ is a (possibly negative) multiple of $e$, whereas the exact order to which it divides $p^{j/e}$ is
$j$. This implies that the exact order to which $\q$ divides $\gamma_j$ is non-zero.
Using this fact for all  places in $U(\theta)$ above $v$ yields 
a lower bound for the height of $\gamma_j$ of the form $H(\gamma_j)\geq p^{1/(2e[K:\IQ])}$, provided $1\leq j\leq e-1$ and $b_j\neq 0$. 

Combining the upper and lower bounds for the height of $\gamma_j$, and using that  $e$ is  bounded in terms of  $d$, gives an upper bound $B(K,d,X)$ for $p$ in terms
of $d$, $[K:\IQ]$, and $X$, whenever one among $b_1,\ldots,b_{e-1}$ is non-zero. 

This means that for each place $v$ of $K$ lying above a prime $p>B(K,d,X)$ we have  $\alpha\in U$, and $v$ is unramified in $U$.
Therefore, $K(\alpha)$ is unramified at each prime $p$ whenever $p>B(K,d,X)$ (assuming, as we can, $B(K,d,X)>|D_{K/\IQ}|$).
But, by Lemma \ref{lem: basic}, the largest prime $p$ ramifying in $K(\alpha)$ tends to infinity when $K(\alpha)$ runs over an infinite set of subfields of  $K^{(d)}_{ab}$.
Hence, we conclude that $\alpha$ lies in a number field, depending only on $K$, $d$ and $X$,
and thus, by Northcott's Theorem, there are only finitely many possibilities for $\alpha$. 
This completes the proof.

%But, by Lemma \ref{lem: basic}, the largest prime $p_M$ ramifying in $M$ tends to infinity as $M$ runs over all number fields in $K\subset M\subset K^{(d)}_{ab}$.
%Hence, we conclude that $\alpha$ lies in a number field, depending only on $K$, $d$ and $X$,
%and thus, by Northcott's Theorem, there are only finitely many possibilities for $\alpha$. 
%This completes the proof.
\end{proof}

%\begin{remark}

The first proof of Theorem \ref{thm: 1}  (using Theorem \ref{prop:1}) only requires  $e$ to be  bounded in terms of $d$, whereas
the second proof above requires the ramification index $e$ to divide $d!$ to conclude that $L(\theta)/K$ is Galois (and abelian).

The fact  that $K^{(d)}_{ab}/K$ is abelian is used in both proofs  in three different ways, namely to ensure that:

\begin{itemize}
\item[(i)] $p_M$ in Lemma \ref{lem: basic} tends to infinity,
\item[(ii)] $M/K$ is Galois for every number field  $K\subset M \subset K^{(d)}_{ab}$,
\item[(iii)] the inertia groups $I(\B/\p)$ for the different prime ideals $\B\subset \Oseen_M$ above $\p\subset \Oseen_K$ are all equal.
\end{itemize}

The second claim of Lemma \ref{lem: basic} remains true for $K^{(d)}/K$ (replace $M$ by its Galois closure over $K$ in the proof) but
the proof of the first claim falls apart for $K^{(d)}$ when $d\geq 3$. This is because not all finite extensions of $K$  in $K^{(d)}$ can be written as 
compositum of number fields of uniformly bounded degree over $K$ as was shown by Checcoli \cite[Theorem 1]{Checcoli}, at least if $d\geq 27$.  
Gal and Grizzard \cite[Corollary 1.2]{GalGrizzard} showed that  $d\geq 3$ suffices. 

However, they also showed \cite[Theorem 1.3]{GalGrizzard} that every number field in $K^{(3)}$ that is Galois over $K$ can be written as a compositum of extensions
of $K$ of degree at most $3$. This means that if we only consider $\alpha$ in the set $K^{(3)}_G=\{\alpha\in K^{(3)}; K(\alpha)/K \text{ is Galois}\}$ then (i) and (ii) are automatically satisfied
for each $M=K(\alpha)$. This raises the question whether $K^{(3)}_G$ has the Northcott property. An affirmative answer would be a significant extension of the case $d=3$
in Theorem \ref{thm: 1}.

\bibliographystyle{amsplain}
\bibliography{literature}

\end{document}